\documentclass[12pt,a4paper]{article}
 
\usepackage[utf8]{inputenc}
\usepackage[T1]{fontenc}
\usepackage[english]{babel}
\usepackage{lmodern}
 \usepackage{graphicx}
\usepackage{amsfonts}
\usepackage{amsmath}
\usepackage{amssymb}
\usepackage{amsthm}
\usepackage{geometry}
\newtheorem{thm}{Theorem}

\newtheorem{lem}{Lemma}
\newtheorem{prop}{Proposition}

\newtheorem{rem}{Remark}

\newcommand{\Ocal}{\mathcal{O}}
\newcommand{\TODO}[1]{}
\newcommand{\EXCLUDE}[1]{}
\DeclareMathOperator{\sgn}{sgn}
 
 \newcommand{\mytitle}[1]{\goodbreak\medskip\noindent{\itshape #1}}
 
\begin{document}
\title{Explicit upper bound for $\left|L(1, \chi)\right|$\\ when $\chi(2)=1$ and $\chi$ is even}
 
\author{Sumaia Saad Eddin\thanks{S. Saad Eddin is supported by the Austrian Science Fund (FWF): Project F5507-N26, which is a part of the Special Research Program `` Quasi Monte Carlo Methods: Theory and Applications".}}
 
\maketitle
 \begin{abstract}
 Let $\chi$ be a primitive Dirichlet character of conductor $q$ and let us denote by $L(z, \chi)$ the associated $L$-series. In this paper, we provide an explicit upper bound for $\left|L(1, \chi)\right|$ when $\chi$ is a primitive even Dirichlet character with $\chi(2)=1$.
 \end{abstract}
  
  

 \section{Introduction and results}
 Let $\chi$ be a primitive Dirichlet character of conductor $q$ and let us denote by $L(z, \chi)$ the associated $L$-series. Recall we say that $\chi$ even when $\chi(-1)=1$ and odd when $\chi(-1)=-1.$ The upper bound for $\left|L(1, \chi)\right|$ has received considerable attention near the end of the 20th century, mainly because of the importance of this bound in number theory. Several authors have obtained upper bounds for $\left|L(1, \chi)\right|$ with conditions on the modulus (see for instance \cite{Burgess1966}, \cite{Chowla1964},\cite{Stephens1972}, \cite{Pintz1977}, \cite{Toyoizumi1990}, \cite{Granville2002}, \cite{Booker2006} and the references therein).
 
 Concerning fully explicit estimates, it is known that there exists a constant $C$ such that $\left|L(1, \chi)\right|$ satisfies the following bound 
 \begin{equation}
 \label{equ1}
 \left|L(1, \chi)\right| \leq \tfrac{1}{2}\log q +C \qquad (q>1).
 \end{equation}
 The problem of beating the $\tfrac{1}{2}\log q$ has since been addressed but the only results obtained so far have been under the conditional hypothesis that $\chi(2)$ is noticeably different from $1$ (or if not $\chi(2)$, then $\chi(p)$ for some small prime~$p$), see \cite{Louboutin2004} and \cite{Ramare2004}. The aim of this paper is to study the most difficult case, i.e., when $\chi(2)=1$ and prove that
 \begin{thm}
 \label{Thm1}
 Let $\chi$ be an even primitive Dirichlet character of conductor $q>1$ and suppose that $\chi(2)=1$. Then, we have 
 \begin{equation*}
 \left|L(1, \chi)\right| \leq \tfrac{1}{2}\log q -0.02012.
 \end{equation*}
 \end{thm}
 As an example of application, we deduce an explicit upper bound for the class number of a real quadratic field $\mathbb{Q}\left(\sqrt{q}\right)$, improving on a result by Le~\cite{Le1994}. We prove that
 \begin{thm}
  \label{Thm3}
  For every real quadratic field of discriminant $q> 1$ and $\chi(2)=1$, we have
  \begin{equation*}
   h\left(\mathbb{Q}\left(\sqrt{q}\right) \right)
  \leq 
  \frac{\sqrt{q}}{2}\left(1-\frac{1}{25\log q}\right), 
  \end{equation*}
  where $h\left( \mathbb{Q}(\sqrt{q})\right)$ is the class number of $\mathbb{Q}\left(\sqrt{q}\right)$.
  \end{thm}
 
 Concerning $C$ from Eq. \eqref{equ1}; we note that Louboutin~\cite{Louboutin1996} and \cite{Louboutin2001} used integral representation of Dirichlet $L$-functions and obtained the following upper bounds of $ \left|L(1,\chi)\right|$ for primitive $\chi$ of conductor $q>1.$
 \begin{equation*}
 \left|L(1, \chi)\right|\leq 
 \begin{cases}
    \displaystyle{\tfrac{1}{2}\log q+ 0.009} \qquad  &\textrm{if $\chi(-1)=+1$,}\\
   
   \displaystyle{\tfrac{1}{2}\log q+0.717}\qquad  &\textrm{if $\chi(-1)=-1$.}
 \end{cases}
 \end{equation*}
 In 2001, Ramar\'e~\cite{Ramare2001} gave new approximate formulae for $L(1, \chi)$ depending on Fourier transforms. Thanks to these formulae, this author proved that
 \begin{equation*}
 \left|L(1, \chi)\right|\leq 
 \begin{cases}
    \displaystyle{\tfrac{1}{2}\log q} \qquad  &\textrm{if $\chi(-1)=+1$,}\\
   
   \displaystyle{\tfrac{1}{2}\log q+0.7083}\qquad  &\textrm{if $\chi(-1)=-1$. }
 \end{cases}
 \end{equation*}
 Using numerical evidence, Ramar\'e proposed the conjecture 
 \begin{equation*}
 \left\{
 \begin{array}{ll}
 \max\limits_{\chi \text{ even}} \left \{\left|L(1, \chi) \right|-\tfrac{1}{2}\log q\right \}\mathrel{\mathop{=}\limits^{\text{?}}}-0.32404\cdots, \\
 \max\limits_{\chi \text{ odd}} \{\left|L(1, \chi) \right|-\tfrac{1}{2}\log q \}\mathrel{\mathop{=}\limits^{\text{?}}}0.51482\cdots.
 \end{array}
 \right.
 \end{equation*} 
 The first one being reached by a character modulo $241$ and the second one by a character modulo $311$. Recently and using a
 very refined algorithm (see~\cite{Platt2011}), David Platt has checked this conjecture for all primitive $\chi$ of conductor $2\leq q\le 2\,000\,000$. David Platt has kindly agreed to run again his algorithm to confirm that our result (Theorem~\ref{Thm1}) also holds for conductor $2\leq q\le 2\,000\,000$. The author would like to thank David Platt from Bristol University for his help concerning the computations of this paper. Thank to Olivier Ramar\'e for his helpful
 comments concerning the material of this article.
 
 In the next section, we present the structure of the proof of Theorem~\ref{Thm1}.
 \section{Proof structure of Theorem \ref{Thm1}}
 
 \begin{thm}
 \label{Thm2}
 Let $\chi$ be a primitive Dirichlet character of conductor $ q>1 $. Let $ F : \mathbb{R} \longrightarrow \mathbb{R}$ be such that $F(t)/t$ is in $\mathcal{C}^2(\mathbb{R})$ (also at $0$), vanishes at $t=\pm \infty$ and its first and second derivatives belong to $\mathcal{L}^{1}(\mathbb{R}).$  Assume further that $F$ is even if $\chi$ is odd and that $F$ is odd if $\chi$ is even. We define 
  \begin{equation}
  \label{eq0011}
  G_{F}(u)=\sum_{\ell\geq 0}\frac{1-F(2^\ell u)}{2^\ell}\chi(2)^{\ell}.
  \end{equation}
  Suppose that $G_{F}(u)$ and the derivative of $(G_{F}(u)-2)/u$ are positive for $ u_0\leq u\leq 1$, that $|1-F(t)|\leq c_0/t^2$ , $|F^{\prime}(t)|\leq c_1/t^2$, $|F(t)|\leq c_2$, $|F(t)|\leq c_3 t$ and $|G_F(t)|\leq 2$ for all $t>0$. Then for any $\delta >0$, we have 
 \begin{equation*}
 \sum_{\substack{ m\geq 1\\ (m,2)=1}}\frac{\left|G_{F}(\delta m)\right|}{m} 
 \leq 
 -\log \delta +b_F+D(\theta) +H(\delta),
 \end{equation*}
 where 
 \begin{equation}
 \label{eq11}
 D(\theta)=\frac{c_3\theta}{2\log 2}\left( -\log \theta+\log \tfrac{c_2}{c_3}+\log \log 2+\log 2+2\right),
 \end{equation}
 and 
 \begin{equation*}
 H(\delta)=\log(1-\delta)
      +\frac{(4c_0+7c_1)\delta}{21}
      +\frac{(12c_0+14c_1)\delta ^{2}}{21}
     +\frac{\delta^2}{6(1-4\delta)^2}
     +\frac{\delta^4}{30(1-4\delta)^4}.
 \end{equation*}
 Here, $\theta$ is $ u_0$ or $\delta$ according to whether $\delta \leq u_0$ or $\delta >u_0$. The constant $b_F$ is given by  
 \begin{equation}
 \label{eqe3}
 b_F= \frac{1}{2} \int_{0}^{1}\frac{G_{F}(t)-2}{t}\, dt
   + 
       \frac{1}{2}\int_{1}^{\infty}\frac{|G_{F}(t)|}{t}\, dt
       +\gamma+\log 2,
 \end{equation}
 where $\gamma$ denote Euler's constant.
 \end{thm}
 The proof of this theorem is long and complicated. It requires several tools. We begin by collecting some important results proven in~\cite{Ramare2001}.
 \begin{lem}
  \label{ramare3}
  We have 
  \begin{equation*}
  -\sum\limits_{1\leq k\leq \delta q/2}\log \left|\sin \frac{\pi k}{\delta q}\right|
  \leq 
  \frac{q \delta }{2}\log 2.
  \end{equation*}
  \end{lem}  
 \begin{prop}
 \label{pro62}
 Set 
 \begin{equation}
 \label{eq12}
 F_1(t)=\frac{\sin (\pi t)}{\pi}\left(\log 4+\sum\limits_{n\geq 1}(-1)^n\left(\frac{2n}{t^2-n^2}+\frac{2}{n}\right)\right).
 \end{equation}
 Let $\chi$ be an even primitive Dirichlet character of conductor $q>1$. Then, we have
 \begin{equation}
 \label{eq13}
 L(1, \chi)= 
  \sum\limits_{n\geq 1} \frac{\left(1-F_1(\delta n)\right)\chi(n)}{n}-\frac{2\tau(\chi)}{q}\sum\limits_{1\leq m\leq \delta q/2}\overline{\chi }(m) \log \left|\sin \frac{\pi m}{\delta q}\right|.
 \end{equation}
 \end{prop}
 We build our idea to prove Theorem \ref{Thm1} on the fact that the function
 $1-F_1$ is not positive on $\mathbb{R}_+$. Thus, part of loss of the upper
 bound for $\left|L(1, \chi)\right|$ is due to studying
 $\sum_{n}\left|1-F_1(\delta n)\right|/n$ rather than $\sum_{n}\left(1-F_1(\delta n)\right)/n$. In this paper, we study  $\sum_{n}\left(1-F_1(\delta n)\right)/n$ directly. To do so, we need to observe some properties of the functions $F_1$ and then $G_{F_1}$. The following two sections state the most important results concerning these functions. 
 \section{Study of $F_1$ }
  
  \begin{lem}
  \label{lemm64}
  The function $F_1(t)$ satisfies 
 \begin{equation*}
 0 \leq \sgn \left(\sin \pi t\right)\left \{ \sgn (t)- F_1(t)\right \}\leq \left|\frac{\sin \pi t}{\pi t}\right|\frac{1}{1+|t|}.
  \end{equation*}
  \end{lem}
  \begin{proof}
  See \cite{Vaaler1985}.
  \end{proof}
  A simple consequence of this lemma is the following result.
  \begin{lem}
   \label{lemm65}
   For $t$ real number, we have 
 \begin{equation*}
   0 \leq \sgn \left(\sin(\pi t)\right)\left(1-F_1(t)\right) 
   \leq
    \left|1/ (\pi t)\right|, 
    \qquad
     \left|F^{\prime}(t)\right| \ll t^{-2}.
     \end{equation*}
   \end{lem}
   \begin{proof}
   See \cite[Lemma 12]{Ramare2001}.
   \end{proof}
  \begin{lem}
    \label{lemm66}
    For $t\geq 0$, we have $|F_1(t)|\leq (\log 4) t.$
    \end{lem}
  \begin{proof}
 From \cite{Ramare2001}, we recall that 
  \begin{eqnarray}
  \label{eqe2}
  \frac{F_1(t)}{t}
  &=&
  -2\int_{-1/2}^{1/2}\log \left|\sin(\pi v)\right| e(-tv)\, dv \nonumber
 \\& =&
 -4\int_{0}^{1/2}\log \left|\sin(\pi v)\right| \cos (2\pi tv)\, dv. 
  \end{eqnarray}
   It follows that 
  \begin{equation*}
   \frac{|F_1(t)|}{t}
  \leq 
     4\int_{0}^{1/2}\log \left|\sin(\pi v)\right| \, dv=\log 4.
  \end{equation*}
  This completes the proof.
  \end{proof}
 \section{Study of $G_{F_1}$}
 
 \begin{lem}
 \label{lem8}
 For $ u \geq 0$, we have $\left|G_{F_1}(u)\right|\leq 2$.
 \end{lem}
 \begin{proof}
 From Lemma \ref{lemm64}, we get
 \begin{equation*}
 \left|1-F_1(2^\ell u)\right| \leq \frac{|\sin(2^\ell \pi u)|}{2^\ell \pi u} \frac{1}{1+2^\ell u}\leq 1.
 \end{equation*}
 It follows that 
 \begin{equation*}
 \left|G_{F_1}(u)\right|
 \leq 
 \sum_{\ell \geq 0}\frac{\left|1-F_1(2^\ell u)\right|}{2^\ell}
 \leq 
 \sum_{\ell \geq 0}\frac{1}{2^\ell} =2.
 \end{equation*}
 This completes the proof.
 \end{proof}
 \begin{lem}
 \label{lem6}
 For $10^{-5}\leq u\leq 1$, $G_{F_1}(u)$ is positive.
 \end{lem}
 \begin{proof}
 Recall that
 \begin{equation}
 \label{eq14}
 G_{F_1}(u)=\sum_{\ell \geq 0}\frac{1-F_1(2^\ell u)}{2^\ell}
 =2-\sum_{\ell \geq 0}\frac{F_1(2^\ell u)}{2^\ell}.
 \end{equation}
 From Eq. \eqref{eqe2}, we write
 \begin{equation}
 \label{eq15}
 \frac{F_1(t)}{t}=\frac{2}{t}\int_{0}^{1/2}\cot (\pi v) \sin (2\pi tv)\, dv.
 \end{equation}
 Using integration by parts, we get
 \begin{equation}
 \label{eqe4}
 G^{\prime}_{F_1}(u)
 =
 \frac{2}{u} \sum_{\ell \geq 0} \int_{0}^{1/2} \left(\cot (\pi v)-\frac{\pi v}{\sin^2 (\pi v)}\right) \frac{\sin(2^{\ell +1}\pi uv)}{2^\ell}\, dv.
 \end{equation}
 Now, we define, for an integer parameter $L$, 
 \begin{equation*}
  S_L(u)= \sum_{\ell \leq L} \int_{0}^{1/2} \psi(v) \frac{\sin(2^{\ell +1}\pi uv)}{2^\ell}\, dv,
  \end{equation*}
  and 
  \begin{equation*}
 R_L(u)= \sum_{\ell \geq L+1}\int_{0}^{1/2} \psi(v) \frac{\sin(2^{\ell +1}\pi uv)}{2^\ell}\, dv,
 \end{equation*}
 where the function $\psi(v)=\cot (\pi v)-(\pi v)/\sin^2(\pi v)$ is negative. Eq. \eqref{eqe4} can be written by
 \begin{equation}
 \label{eq62}
 G^{\prime}_{F_1}(u)
 = S_L(u)+R_L(u).
 \end{equation}
 Notice that
 \begin{equation*}
 \left|R_L(u)\right|
 \leq 
 \sum_{ \ell \geq L+1}\int_{0}^{1/2^{\ell+1}} \left| \psi (v)\right| 2\pi uv\, dv +\sum_{ \ell \geq L+1} \int_{1/2^{\ell+1}}^{1/2} \frac{\left| \psi (v)\right|}{2^{\ell}}\, dv.
 \end{equation*}
 On using $\left|\sin(2\pi v)-2\pi v\right|\leq (2\pi v)^3/3!$ and $\sin (\pi v)\geq 2v $ for all $0~\leq~v~\leq~1/2$, we get $|\psi (v)|~\leq~\pi^3 v/6$. It follows that 
 \begin{eqnarray*}
 \left|R_L(u) \right|
 &\leq &
  \frac{\pi^4u}{9\cdot 8}\sum_{\ell \geq L+1}\frac{1}{2^{3\ell}} + \frac{\pi^3}{6\cdot 8}\sum_{ \ell \geq L+1}\left(\frac{1}{2^{\ell}}-\frac{1}{2^{3\ell}}\right)
 \\ &\leq &
  \frac{\pi^4 u}{9\cdot 8\cdot 7\cdot 2^{3L}}+\frac{\pi^3}{ 6\cdot 8}\left(\frac{1}{2^{L}}-\frac{1}{7\cdot 2^{3L}}\right).
 \end{eqnarray*}
 Therefore
 \begin{equation*}
 \left|R_L(u)\right|
 \leq 
   \frac{\pi^4 u}{504\cdot 2^{3L}}+\frac{\pi^3}{48}\left(\frac{1}{2^{L}}-\frac{1}{7\cdot 2^{3L}}\right).
 \end{equation*}
 For $10^{-5}\leq u\leq 1$ and $L=15$, the maxima of $S_L(u)+\left|R_L(u)\right|$ numerically seems non-increasing and GP/PARI needs at most $10$ seconds to prove it is $\leq  -0.0001353$. Since $G_{F_1}(1)=0$, it follows that $G_{F_1}$ is positive. This completes the proof.
 \end{proof}
  \begin{lem}
  \label{lem7}
 For $ 10^{-5}\leq u\leq 1$, we have 
 $\frac{d}{du}\left(\frac{G_{F_1}(u)-2}{u}\right)\geq 0.$
 \end{lem}
 \begin{proof}
 Since 
 \begin{equation*}
 \frac{G_{F_1}(u)-2}{u}=-\sum\limits_{\ell \geq 0}\frac{F_1\left(2^{\ell} u\right)}{2^{\ell}u}, 
 \end{equation*}
 using Eq. \eqref{eq15}, we obtain
 \begin{multline*}
 \frac{d}{du}\left(\frac{G_{F_1}(u)-2}{u}\right)
 =
 \frac{2}{ u^2}\sum_{\ell \geq 0}\frac{1}{2^\ell}\int_{0}^{\frac{1}{2}} \cot(\pi v) \sin(2^{\ell+1}\pi uv)\, dv
 \\-\frac{4\pi }{u}\sum_{\ell \geq 0}\int_{0}^{\frac{1}{2}} v\cot(\pi v) \cos(2^{\ell+1}\pi uv)\, dv.
 \end{multline*}
 Using integration by parts for the last integral above, we find that
 \begin{equation}
 \label{eq71}
 \frac{d}{du}\left(\frac{G_{F_1}(u)-2}{u}\right)
 =
 \frac{4}{u^2}\sum_{\ell \geq 0}\int_{0}^{\frac{1}{2}} \frac{\sin(2^{\ell+1}\pi uv)}{2^\ell}\varphi (v)\, dv.
 \end{equation}
 where $\varphi(v)=\cot(\pi v)-\pi v/(2\sin^2 (\pi v))$.
 For an integer parameter $L$, we define
 \begin{equation*}
 \tilde{S}_L(u)  = \sum_{0\leq \ell \leq L} \int_{0}^{1/2} \varphi (v) \frac{\sin (2^{\ell +1}\pi uv)}{2^{\ell}}\, dv,
 \end{equation*}
 and 
 \begin{equation*}
 \tilde{R}_L(u)  = \sum_{\ell \geq L+1} \int_{0}^{1/2} \varphi (v) \frac{\sin (2^{\ell +1}\pi uv)}{2^{\ell}}\, dv.
 \end{equation*}
 Then, Eq. \eqref{eq71} can be written as 
 \begin{equation}
 \label{eq72}
 \frac{d}{du}\left(\frac{G_{F_1}(u)-2}{u}\right)
 =
 \tilde{S}_L(u)+\tilde{R}_L(u).
 \end{equation}
 Notice that
 \begin{equation*}
 \left|\tilde{R}_L(u)\right|
 \leq 
 \sum_{ \ell \geq L+1}\int_{0}^{1/2^{\ell+1}} \left| \varphi (v)\right| 2\pi uv\, dv +\sum_{ \ell \geq L+1} \int_{1/2^{\ell+1}}^{1/2} \frac{\left| \varphi (v)\right|}{2^{\ell}}\, dv.
 \end{equation*}
 On using $\left|\sin(2\pi v)-\pi v\right|\leq \pi v$ and $\sin (\pi v)\geq 2v $, for all $0\leq v\leq 1/2$, we get $|\varphi (v)|\leq \pi/(8 v)$. It follows that: 
 \begin{eqnarray*}
 \left|\tilde{R}_L(u)\right|
 &\leq &
  \frac{\pi^2u}{8}\sum_{\ell \geq L+1}\frac{1}{2^{\ell}} + \frac{\pi \log 2}{8}\sum_{ \ell \geq L+1}\frac{\ell}{2^{\ell}}
 \\ &\leq &
  \frac{\pi^2 u}{2^{L+3}}+\frac{\pi (L+2)\log 2}{ 2^{L+3}}.
 \end{eqnarray*}
 Therefore
 \begin{equation*}
 \left|\tilde{R}_L(u)\right|
 \leq 
  \frac{\pi^2 u}{2^{L+3}}+\frac{\pi (L+2)\log 2}{ 2^{L+3}}.
 \end{equation*}
 For $10^{-5}\leq u\leq 1$ and $L=21$, the minima of $\tilde{S}_L(u)+\left|\tilde{R}_L(u)\right|$ numerically seems increasing and GP/PARI needs at most $10$ seconds to prove it is $\geq 0.0000019$. Then, the derivative of $\left(G_{F_1}-2\right)/u$ is positive. This completes the proof.
 \end{proof}
 \begin{lem}
 \label{lem5}
 Set
 \begin{equation}
 \label{equ62}
 b_{F_1}= \frac{1}{2} \int_{0}^{1}\frac{G_{F_1}(t)-2}{t}\, dt
   + 
       \frac{1}{2}\int_{1}^{\infty}\frac{|G_{F_1}(t)|}{t}\, dt
       +\gamma+\log 2.
 \end{equation}
 Then $b_{F_1}  \leq -0.66266.$
 \end{lem}
 \begin{proof}
 By Eq. \eqref{eq12}, we have 
  \begin{equation*}
  F_{1}(t)=\frac{\sin (\pi t)}{\pi}\left\{\log 4+\sum_{n\geq 1}(-1)^n\left(\frac{1}{t-n}+\frac{1}{n}\right)-\sum_{n\geq 1}(-1)^n\left(\frac{1}{t+n}-\frac{1}{n}\right)\right\}.
   \end{equation*}
 Thanks to the equality $(2.2)$ of \cite{Vaaler1985}, when $F=1$, we write 
 \begin{equation*}
 \frac{\pi}{\sin (\pi t)}=\frac{1}{t}+\sum_{n\geq 1}(-1)^n\left(\frac{1}{t-n}+\frac{1}{n}\right)+\sum_{n\geq 1}(-1)^n\left(\frac{1}{t+n}-\frac{1}{n}\right). 
 \end{equation*}
 It follows that 
 \begin{equation*}
 F_{1}(t)
  =
 \frac{\sin (\pi t)}{\pi}\left\{\log 4+\frac{\pi}{\sin (\pi t)}-\frac{1}{t}-2\sum_{n\geq 1}(-1)^n\left(\frac{1}{t+n}-\frac{1}{n}\right)\right\}.
 \end{equation*}
 Thus
 \begin{equation}
 \label{eqe6}
 F_{1}(t)=
 1+\frac{\sin (\pi t)}{\pi}\left\{\log 4-\frac{1}{t}-2\sum_{n\geq 1}(-1)^n\left(\frac{-t}{n(t+n)}\right)\right\}.
 \end{equation}
 From Eqs. \eqref{equ62} and \eqref{eqe6}, we write the following simple GP-PARI code:
 \begin{verbatim}
 {F1(t) = 
    1+sin(Pi*t)/Pi*(log(4)-1/t+2*t*sumalt(n=1,cos(Pi*n)/n/(t+n)))}
 {G(x, borne=50) = sum(l=0, borne, (1-F1(2^l*x))/2^l)}
 {GG(x, borne=50) = sum(l=0, borne,-F1(2^l*x)/2^l)}
 {bF1( borne=50, bornex=1000) = 
     Euler+log(2)+1/2*intnum(x=0, 1, GG(x, borne)/x)
      +1/2*intnum(x=1, bornex, abs(G(x, borne))/x)}
  default(realprecision,200)
 \end{verbatim}
 For $T=1000$ and $L=50$, we find that 
 \begin{equation*}
 \left|-\frac{1}{2}\int_{0}^{1}\sum_{\ell \geq L+1}\frac{F_1(2^\ell t)}{2^\ell t}\, dt +\frac{1}{2}\int_{1}^{T}\sum_{\ell \geq L+1}\frac{1-F_1(2^lt)}{2^lt}\, dt +\frac{1}{2}\int_{T}^{\infty}\sum_{\ell \geq 0}\frac{1-F_1(2^\ell t)}{2^lt}\, dt \right|
 \leq 0.00022,
 \end{equation*}
 and that $b_{F_{1}}\leq-0.66266.$ 
 This completes the proof.
  \end{proof}
 \begin{lem}
 \label{lem1}
 Under the hypotheses of Theorem \ref{Thm2}, we have  
 \begin{equation*}
 \max_{T\leq u}\left|\frac{d}{du} \left(\frac{G_F(u) }{u}\right)\right|
 \leq 
 \frac{8c_0}{7T^4}+\frac{4c_1}{3T^3}.
 \end{equation*}
 where the constants $c_0$ and $c_1$ are positive numbers chosen according to the function $F$.
 \end{lem}
 \begin{proof}
 Notice that
 \begin{eqnarray*}
 \left|\frac{d}{du}\left(\frac{G_{F}(u)}{u}\right)\right|\leq \frac{|G^\prime _{F}(u)|}{u}+\frac{|G_{F}(u)|}{u^2}.
 \end{eqnarray*}
 From our hypotheses, it follows that the series $\sum_{\ell\geq 0}\frac{1-F(2^{\ell}u)}{2^{\ell} u^2}$ converges normally as well as the series $\sum_{\ell \geq 0}\frac{\left(1-F(2^{\ell}u)\right)^{\prime}}{2^{\ell}u}$ uniformly for $u \in [\lambda, \infty[$, for any $\lambda >0$. We infer that
 \begin{equation*}
 \left|\frac{G^\prime_{F}(u)}{u}\right|\leq \sum_{\ell \geq 0}\frac{\left|F^{\prime}(2^{\ell}u)\right|}{u} \leq \sum_{\ell \geq 0}\frac{c_1}{2^{2\ell}u^3}=\frac{4c_1}{3u^3},
 \end{equation*}
 and
 \begin{equation*}  
 \frac{|G_{F}(u)|}{u^2}\leq \sum_{\ell \geq 0}\frac{|1-F(2^{\ell}u)|}{2^{\ell}u^2}
 \leq 
 \sum_{\ell \geq 0}\frac{c_{0}}{2^{3\ell}u^4}
  =
  \frac{8c_{0}}{7u^4}.
 \end{equation*}
 This completes the proof.
 \end{proof}
 \begin{lem}
  \label{lem2}
 Under the hypotheses of Theorem \ref{Thm2}. we have 
 \begin{equation*}
  \int_{0}^{u_0}\left|\frac{G_F(u)-2}{u}\right|\, du
  \leq 
  \frac{c_3u_0}{\log 2}\left( -\log u_0+\log \tfrac{c_2}{c_3}+\log \log 2+\log 2+2\right),
 \end{equation*}
 where $c_2$ and $c_3$ are positive numbers and chosen according to the function $F$.
  \end{lem}
  \begin{proof}
 For $L\geq 0$, we write  
  \begin{equation*}
  \frac{G_{F}(u)-2}{u}=-\sum_{\ell \leq L}\frac{F(2^\ell  u)}{2^\ell }
  -\sum_{\ell \geq L+1}\frac{F(2^\ell  u)}{2^\ell}.
  \end{equation*} 
  Since $|F(t)|\leq c_2$  and $|F(t)|\leq c_3 t$, we get 
  \begin{equation*}
   \left|\sum_{\ell \geq 0}\frac{F(2^\ell  u)}{2^\ell }\right|
 \leq
   c_3(L+1)u+\frac{c_2}{2^L}.
   \end{equation*} 
  The best value of $L$ is given by $2^L=\frac{c_2\log 2}{c_3 u}.$
  Then, we find that 
 \begin{equation*}
   \left|\sum_{\ell \geq 0}\frac{F(2^\ell u)}{2^\ell }\right|
   \leq
  \frac{c_3 u}{\log 2}\left( 1+\log \tfrac{c_2}{c_3}+\log \log 2+\log 2-\log u\right).
 \end{equation*}
 This yields
 \begin{equation*}
 \int_{0}^{u_0 } \left|\frac{G_F(u)-2}{u}\right|\, du
 \leq 
 \frac{c_3u_0}{\log 2}\left( -\log u_0+\log \tfrac{c_2}{c_3}+\log \log 2+\log 2+2\right),
 \end{equation*}
 which completes the proof.
  \end{proof}
 Now, we record the following simple lemmas, but useful results to complete the proof of Theorem \ref{Thm2}. 
 \section{Preliminary lemmas}
  \begin{lem}
  \label{lem3}
  For any even integer $M\geq 1$, we have 
  \begin{equation*}
    \sum_{\substack{ m\leq M\\ (m,2)=1}}\frac{1}{m} 
   =
   \frac{1}{2}\left(\log M+\gamma+\log 2\right)+\frac{1}{12M^2}+\frac{2\Theta_M}{15M^4},
  \end{equation*}
  where $\Theta_M \in[-1,1/8]$.
  \end{lem}
  \begin{proof}
  This follows immediately from 
  \begin{equation*}
  \sum_{m\leq M}\frac{1}{m}-\log M=\gamma+\frac{1}{2M}-\frac{1}{12M^2}+\frac{\theta_M}{60M^4},
  \end{equation*}
   where $\gamma$ is the Euler constant and $\theta_{M}\in [0,1]$.
  \end{proof}
  \begin{lem}
  \label{lem4}
  For $M\geq 1$, we have 
 \begin{equation*}
 \sum_{\substack{ m\geq M \\ (m,2)=1}}\frac{1}{m^3}=\frac{1}{4M^2}+{\Ocal}^*\left(\frac{1}{2M^3}\right)
 \end{equation*}  
 and 
 \begin{equation*}
 \sum_{\substack{ m\geq M\\ (m,2)=1}}\frac{1}{m^4}=\frac{1}{6M^3}+{\Ocal}^*\left(\frac{1}{2M^4}\right)
 \end{equation*}
 \end{lem}
  \begin{proof}
 Notice that 
 \begin{eqnarray*}
   \sum_{\substack{ m\geq M\\(m,2)=1}}\frac{1}{m^3}
 &=& 
     3\sum_{\substack{ m\geq M\\(m,2)=1}}\int_{m}^{\infty}\frac{dt}{t^4}
 =
     3\int_{M}^{\infty}\sum_{\substack{ M\leq m\leq t\\(m,2)=1}}1 \, \frac{dt}{t^4}
 =
      \frac{3}{2}\int_{M}^{\infty}\frac{[t]-[M]}{t^4}\, dt
 \\&=&
     \frac{3}{2}\left(\int_{M}^{\infty}\frac{dt}{t^3}
     -\int_{M}^{\infty}\frac{\{t\}}{t^4}\, dt
     -M\int_{M}^{\infty}\frac{dt}{t^4}+\{M\}\int_{M}^{\infty}\frac{dt}{t^4}\right)
 \end{eqnarray*}
 As usual, $\{x\}$ and $[x]$ denote the fractional and the integer part of real $x$. Since $0\leq \{x\}<1$, the first equality follows from 
 \begin{equation*}
 0\leq \int_{M}^{\infty}\frac{\{t\}}{t^4}\, dt
 \leq 
 \frac{1}{3M^3}, \qquad
 0\leq \int_{M}^{\infty}\frac{\{M\}}{t^4}\, dt 
 \leq \frac{1}{3M^3}.
 \end{equation*}
 By a similar argument, we prove the second equality of this Lemma.
  \end{proof}
 \section{{Proofs}}
 
 \subsection{Proof of Theorem \ref{Thm2}}
 First, writing $n$ in the form $2^\ell m$ where $m$ is odd integer and recalling the suitable condition $\chi(2)=1$. We have 
  \begin{equation}
  \label{eqthm1}
 \sum_{n\geq 1}\chi(n)\frac{1-F(\delta n)}{n}
 =
 \sum_{\substack{m\geq 1\\ (m,2)=1}}\frac{\chi(m)}{m}\sum_{\ell\geq 0}\frac{1-F(\delta 2^{\ell}m)}{2^{\ell}}=
 \sum_{\substack{ m\geq 1\\ (m,2)=1}} \frac{\chi(m)}{m}G_{F}(\delta m),
  \end{equation}
  where the function $F$ satisfies our hypotheses.
 Now, we compare this sum to an integral. Here we need to consider two cases
 according to the parity of $[\delta^{-1}]$. 
 \begin{itemize}
 \item[(1)] The first case when $[\delta^{-1}]$ is odd in which case we set
   $M=[\delta^{-1}]$.
 \item[(2)] In the second case, $[\delta^{-1}]$ is even and we set $M=[\delta^{-1}-1]$.
  \end{itemize}
 Both treatments are very similar, so we give the details only in the first case.
 We thus assume that $[\delta^{-1}]$ is odd and set $M=[\delta^{-1}]$. We have
 \begin{equation}
 \label{eqthm2}
 \delta \sum_{\substack {m\geq M\\ (m,2)=1}}\frac{|G_{F}(\delta m)|}{\delta m}
  =
 \frac{1}{2}\int_{\delta M}^{\infty}\frac{|G_{F}(t)|}{t}\, dt 
  -
 \frac{1}{2}\sum_{\substack{ m\geq M\\ (m,2)=1}}\int_{\delta m}^{\delta (m+2)}\left(\frac{|G_{F}(t)|}{t}-\frac{|G_{F}(\delta m)|}{\delta m}\right) \, dt.
 \end{equation}
 Concerning the inner integral on the far right-hand side above, we have
 \begin{equation*}
 \left|\frac{|G_{F}(\delta m)|}{\delta m}-\frac{|G_{F}(t)|}{t}\right| 
 \leq  
 (t-\delta m)\max_{ \delta m\leq u }\left|\frac{d}{du} \frac{G_F(u)}{u}\right|.
 \end{equation*}
 Define
 \begin{equation*}
 K_M(\delta)=\sum_{\substack{ m\geq M\\ (m,2)=1}}\int_{\delta m}^{\delta (m+2)}\left|\frac{|G_{F}(t)|}{t}-\frac{|G_{F}(\delta m)|}{\delta m}\right| \, dt.
 \end{equation*}
 Applying Lemma \ref{lem1} and using the identity
 $
 \int_{\delta m}^{\delta(m+2)}(t-\delta m)\, dt = 2\delta ^2$,
 we get 
 \begin{eqnarray*}
 K_M(\delta) 
 &\leq& 
 \sum_{\substack{ m\geq M\\ (m,2)=1}}\int_{\delta m}^{\delta (m+2)}\left(\frac{8c_0}{7\delta^4 m^4}+\frac{4c_1}{3\delta ^3 m^3}\right)(t-\delta m)\, dt
 \\&\leq&
 \frac{16c_0}{7\delta^2}\sum_{\substack{ m\geq M\\ (m,2)=1}}\frac{1}{m^4}
 +
 \frac{8c_1}{3\delta}\sum_{\substack{ m\geq M\\ (m,2)=1}}\frac{1}{m^3}.
 \end{eqnarray*}
 Applying Lemma \ref{lem4} to the sum on the right-hand side above, we find that
 \begin{equation*}
 \frac{1}{2}K_M(\delta)
 \leq 
 \frac{8c_0}{7\delta^2}\left(\frac{1}{6M^3}+\frac{1}{2M^4}\right)
 +\frac{4c_1}{3\delta }\left(\frac{1}{4M^2}+\frac{1}{2M^3}\right). 
 \end{equation*}
 Recall that $M=[\delta ^{-1}]$, Eq. \eqref{eqthm2} becomes
 \begin{equation}
 \label{eqthm4}
          \sum_{\substack{m\geq M\\ (m,2)=1}}\frac{|G_{F}(\delta m)|}{m}
  \leq
         \frac{1}{2}\int_{1}^{\infty}\frac{|G_{F}(t)|}{t}\, dt +H_1(\delta),
 \end{equation}
 where 
 \begin{equation}
 \label{eqe7}
 H_1(\delta)= \frac{(4c_0+7c_1)\delta}{21}+\frac{(12c_0+14c_1)\delta ^{2}}{21}
 \end{equation}
 Now, let $ u_0\leq \delta m\leq 1$ for the remaining $m$'s, we notice that $G_F(\delta m)\geq 0$. Here, we distinguish two cases.
 \begin{itemize}
 \item[(a)] \textbf{Assume that $u_0/\delta<1 $}. We have
 \begin{equation}
 \label{eqthm71}
         \sum_{\substack{ 1\leq m\leq M-2\\(m,2)=1}}\frac{|G_{F}(\delta m)|}{ m}
  =
  \sum_{\substack{ 1 \leq m\leq M-2\\ (m,2)=1}}\frac{G_{F}(\delta m)-2}{m}
   +\sum_{\substack{ 1\leq m\leq M-2\\ (m,2)=1}}\frac{2}{m}.
 \end{equation}
 We write the first sum on the right-hand side above as
 \begin{multline*}
     \sum_{\substack{ 1\leq m\leq M-2\\ (m,2)=1}}\frac{G_{F}(\delta m)-2}{m} 
 =
     \frac{1}{2}\int_{\delta}^{\delta M}\frac{G_{F}(t)-2}{t}\, dt \\
     -\frac{1}{2}\sum_{\substack{1\leq m\leq M-2\\ (m,2)=1}}\int_{\delta m}^{\delta (m+2)}\left(\frac{G_{F}(t)-2}{t}-\frac{G_{F}(\delta m)-2}{\delta m}\right) \, dt.
 \end{multline*}
 Since the derivative of $(G_F(u)-2)/u$ is positive on $u_0\leq u\leq 1$ and $u_0<\delta$. Then the derivative $(G_F(u)-2)/u$ is also positive on $\delta\leq u\leq 1$, it follows that the last sum above is positive. Applying Lemma \ref{lem3} to the last sum on the right-hand side of Eq. \eqref{eqthm71}, we get 
 \begin{multline*}
         \sum_{\substack{ 1\leq m\leq M-2\\(m,2)=1}}\frac{|G_{F}(\delta m)|}{ m}
  \leq
          \frac{1}{2}\int_{\delta }^{\delta M}\frac{G_{F}(t)-2}{t}\, dt +\\
         \log (M-1)
         +\gamma +\log (2)
         +\frac{1}{6(M-1)^2}+\frac{1}{30(M-1)^4}.
 \end{multline*}
 Recalling that $M=[\delta ^{-1}]$, then
 \begin{equation}
 \label{eqthm81}
         \sum_{\substack{ 1\leq m\leq M-2\\(m,2)=1}}\frac{|G_{F}(\delta m)|}{ m}
 \leq  
          \frac{1}{2}\int_{\delta}^{1}\frac{G_{F}(t)-2}{t}\, dt
         -\log \delta +\gamma +\log 2 +H_2(\delta),
 \end{equation}
 where
 \begin{equation}
 \label{eqe8}
 H_2(\delta)=\log(1-\delta)+\frac{\delta^2}{6(1-2\delta)^2}+\frac{\delta^4}{30(1-2\delta)^4}.
 \end{equation}
 From Eqs. \eqref{eqthm4} and \eqref{eqthm81}, we get
 \begin{multline}
 \label{equ86}
 \sum_{\substack{ m\geq 1\\ (m,2)=1}}\frac{\left|G_{F}(\delta m)\right|}{m} 
 \leq 
 \frac{1}{2}\int_{\delta}^{1}\frac{G_{F}(t)-2}{t}\, dt
 +\frac{1}{2}\int_{1}^{\infty}\frac{|G_{F}(t)|}{t}\, dt \\
 -\log \delta +\gamma +\log 2+H_1(\delta)+H_2(\delta).
 \end{multline}
 Now, we write the first integral on the right-hand side above as 
 \begin{equation*}
 \int_{\delta}^{1}\frac{G_{F}(t)-2}{t}\, dt
 =
 \int_{0}^{1}\frac{G_{F}(t)-2}{t}\, dt
 -
 \int_{0}^{\delta}\frac{G_{F}(t)-2}{t}\, dt. 
 \end{equation*}
 Applying Lemma \ref{lem2} to the last integral above. Eq. \eqref{equ86} becomes
 \begin{equation}
 \label{equ87}
 \sum_{\substack{ m\geq 1\\ (m,2)=1}}\frac{\left|G_{F}(\delta m)\right|}{m} 
 \leq 
 -\log \delta +b_F
 +D(\delta)+H_1(\delta)+H_2(\delta),
 \end{equation}
 where the constant $b_F$ depends on only $F$ and is given by Eq. \eqref{eqe3} and $D(\delta)$ is given by Eq. \eqref{eq11}.
 \item[(b)] \textbf{Assume that $u_0/\delta \geq 1$.} We have
 \begin{equation*}
 \sum_{\substack{ 1\leq m\leq M-2\\(m,2)=1}}\frac{|G_{F}(\delta m)|}{ m}
   \leq   
 \sum_{\substack{ 1\leq m< \left[\frac{u_0}{\delta}\right]\\ (m,2)=1}}\frac{\left|G_{F}(\delta m)\right|}{ m}
  +\sum_{\substack{ \left[\frac{u_0}{\delta}\right] \leq m\leq M-2\\ (m,2)=1}}\frac{G_{F}(\delta m)}{m}.
 \end{equation*}
 Under the condition $\left|G_{F}(u)\right|\leq 2$, we obtain that 
 \begin{equation}
 \label{eqe5}
 \sum_{\substack{ 1\leq m\leq M-2\\(m,2)=1}}\frac{|G_{F}(\delta m)|}{ m}
 \leq 
 \sum_{\substack{ 1\leq m\leq M-2\\ (m,2)=1}}\frac{2}{m}+
 \sum_{\substack{\left[\frac{u_0}{\delta}\right]  \leq m\leq M-2\\ (m,2)=1}}\frac{G_{F}(\delta m)-2}{m}. 
 \end{equation}
 We write the second sum on the right-hand side above as 
 \begin{multline*}
     \sum_{\substack{ \left[\frac{u_0}{\delta}\right] \leq m\leq M-2\\ (m,2)=1}}\frac{G_{F}(\delta m)-2}{m} 
 =
     \frac{1}{2}\int_{\delta \left[\frac{u_0}{\delta}\right] }^{\delta M}\frac{G_{F}(t)-2}{t}\, dt
     \\
     -\frac{1}{2}\sum_{\substack{ \left[\frac{u_0}{\delta}\right] \leq m\leq M-2\\ (m,2)=1}}\int_{\delta m}^{\delta (m+2)}\left(\frac{G_{F}(t)-2}{t}-\frac{G_{F}(\delta m)-2}{\delta m}\right) \, dt.
 \end{multline*}
 Since the derivative of $(G_F(u)-2)/u$ is positive on the interval $u_0\leq u\leq 1$, it follows that the last sum above is positive. Applying Lemma \ref{lem5} to the first sum on the right-hand side of Eq. \eqref{eqe5}, we get
 \begin{multline*}
         \sum_{\substack{ 1\leq m\leq M-2\\(m,2)=1}}\frac{|G_{F}(\delta m)|}{ m}
  \leq
          \frac{1}{2}\int_{\delta \left[\frac{u_0}{\delta}\right] }^{\delta M}\frac{G_{F}(t)-2}{t}\, dt \\
         +\log (M-1)
         +\gamma +\log (2)
         +\frac{1}{6(M-1)^2}+\frac{1}{30(M-1)^4}
 \end{multline*}
 Recalling that $M=[\delta ^{-1}]$, then
 \begin{equation}
 \label{eqthm8}
         \sum_{\substack{ 1\leq m\leq M-2\\(m,2)=1}}\frac{|G_{F}(\delta m)|}{ m}
  \leq 
          \frac{1}{2}\int_{\delta \left[\frac{u_0}{\delta}\right] }^{1}\frac{G_{F}(t)-2}{t}\, dt
         -\log \delta 
         +\gamma +\log 2+H_2(\delta)
 \end{equation}
 where $H_2(\delta)$ is given by Eq. \eqref{eqe8}.
 From Eqs. \eqref{eqthm4} and \eqref{eqthm8}, we get
 \begin{multline*}
 \sum_{\substack{ m\geq 1\\ (m,2)=1}}\frac{\left|G_{F}(\delta m)\right|}{m} 
 \leq 
 \frac{1}{2}\int_{\delta \left[\frac{u_0}{\delta}\right] }^{1}\frac{G_{F}(t)-2}{t}\, dt
 +\frac{1}{2}\int_{1}^{\infty}\frac{|G_{F}(t)|}{t}\, dt
 \\
 -\log \delta +\gamma +\log 2+H_1(\delta)+H_2(\delta).
 \end{multline*}
 But, the first integral above is 
 \begin{equation*}
 \frac{1}{2}\int_{\delta \left[\frac{u_0}{\delta}\right] }^{1}\frac{G_{F}(t)-2}{t}\, dt
 =
 \frac{1}{2}\int_{0}^{1}\frac{G_{F}(t)-2}{t}\, dt
 -
 \frac{1}{2}\int_{0}^{\delta \left[\frac{u_0}{\delta}\right] }\frac{G_{F}(t)-2}{t}\, dt.
 \end{equation*}
 Thanks to Lemma \ref{lem2}, we find that
 \begin{equation}
 \label{equ88}
 \sum_{\substack{ m\geq 1\\ (m,2)=1}}\frac{\left|G_{F}(\delta m)\right|}{m} 
 \leq 
 -\log \delta +b_F+D(u_0)+H_1(\delta)+H_2(\delta).
 \end{equation}
 where $D(u_0)$ is given by Eq. \eqref{eq11}.
 \bigskip
 \end{itemize}
 Form Eqs. \eqref{equ87} and \eqref{equ88}, we conclude that
 \begin{equation*}
 \sum_{\substack{ m\geq 1\\ (m,2)=1}}\frac{\left|G_{F}(\delta m)\right|}{m} 
 \leq 
 -\log \delta +b_F+D(\theta)+ H(\delta),
 \end{equation*}
 where $H(\delta)= H_1(\delta)+ H_2(\delta)$ and $\theta$ is $ u_0$ or $\delta$ according to whether $\delta \leq u_0$ or $\delta >u_0.$
 
 \subsection{Proof of Theorem~\ref{Thm1}}
 Setting $n=2^{\ell} m$ with $m$ is odd in Eq. \eqref{eq13}, we have
 \begin{equation}
 \label{thm31}
  L(1,\chi)=\sum_{\substack{ m\geq 1\\(m,2)=1}}\chi (m)\frac{G_{F_{1}}(\delta m)}{m}
               -\frac{2\tau (\chi)}{q}\sum\limits_{1\leq k\leq \delta q/2}\overline{\chi}(k)\log \left|\sin \frac{\pi k}{\delta q}\right|.
 \end{equation}
 From Lemmas~\ref{lem8}, \ref{lem6} and \ref{lem7}, we find that the hypotheses of Theorem \ref{Thm2} are satisfied when $u_0=10^{-5}$. Then, we apply this last Theorem to the first sum in Eq.~\eqref{thm31}. Using Lemma~\ref{ramare3} for the second sum of Eq. \eqref{thm31}, we get
  \begin{equation*}
     |L(1, \chi)|
 \leq  
     -\log \delta+b_{F_1}+\delta \sqrt{q} \log 2+D(\theta)+H(\delta)
 \end{equation*}
 where $\theta$ is $10^{-5} $ or $\delta$ according to whether $\delta \leq 10^{-5}$ or $\delta >10^{-5}$. 
 
 Thanks to Lemmas \ref{lemm65}, \ref{lemm66} and \ref{lem5}, we find that $c_0=1$, $c_1=1/\pi$, $c_2=2$, $c_3=\log 4$ and $b_{F_1}  \leq -0.66266$. The best value for $\delta$ is given by $\delta \sqrt{q} \log 2=1$. Thus, we get
   \begin{equation*}
      |L(1, \chi)|
  \leq  
      \tfrac{1}{2}\log q+\log \log 2-0.66266+1+C(q),
  \end{equation*}
  where the constant $C(q)$ depends only on $q$ and 
  \begin{multline*}
  C(q)=\theta\left( -\log \theta+\log 2+2\right)
       +\log \left(1-\frac{1}{\sqrt{q}\log 2}\right)
       +\frac{(4\pi+7)}{21\pi\sqrt{q}\log 2}\\
       +\frac{(12\pi+14)}{21\pi q \log^2 2}
       +\frac{1}{6(\sqrt{q}\log 2-4)^2}
       +\frac{1}{30(\sqrt{q}\log 2-4)^4}.
  \end{multline*}
 Here, we distinguish the following cases:
 \begin{itemize}
 \item For $2\leq q\leq 2\cdot 10^{6}$. Using the same techniques as described in~\cite{Platt-Saad}, Platt computed $L(1,\chi)$ for all primitive $\chi$ where $\chi(2)=1$ with modulus less than $2\, 000\, 000$. No counter examples to Theorem~\ref{Thm1} were found, the nearest miss being at $q=241$ where $\left| L(1,\chi)\right| - \frac{1}{2}\log q \in [-0.3240421,-0.3240420]$. Figure~\ref{fig1} shows, for each $q$, the largest value of $\left| L(1,\chi)\right| - \frac{1}{2}\log q$ found, and Figure~\ref{fig2} shows the smallest.
 \item For $2\cdot 10^6 \leq q \leq 2\cdot 10^{10}$, we find that $\theta=\delta=1/(\sqrt{q}\log 2)$. It follows that 
 $$C(q)+\log \log 2-0.66266+1 \leq -0.02012.$$ 
 \item For $q\geq 2\cdot 10^{10}$, we find that $\theta\leq 1.02014 \cdot 10^{-5}$. It follows that $$C(q)+\log \log 2-0.66266+1 \leq -0.02012.$$ 
 \end{itemize} 
 Therefore, we conclude that
  \begin{equation*}
 |L(1, \chi)|\leq \tfrac 12\log q-0.02012. 
 \end{equation*}
  This completes the proof.
 \begin{figure}
 \centering
 \includegraphics[width=11cm]{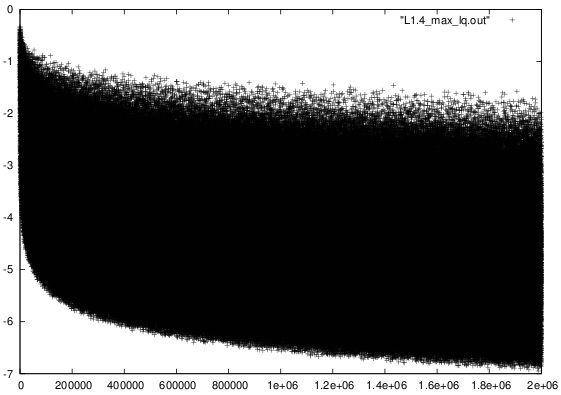}
 \caption{The largest value of $|L(1, \chi)|-\tfrac{1}{2}\log q$.}
 \label{fig1}
 \end{figure}
  \begin{figure}
   \centering
   \includegraphics[width=11.5cm]{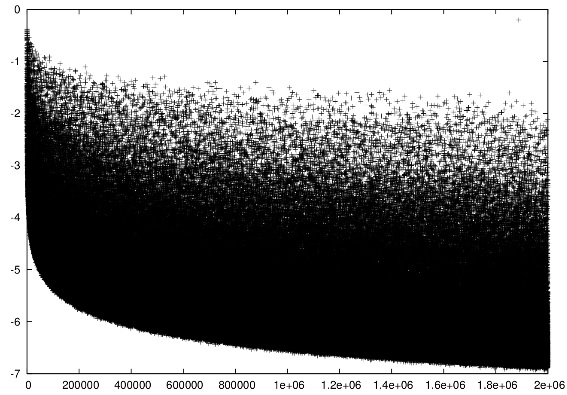}
   \caption{The smallest value of $|L(1, \chi)|-\tfrac{1}{2}\log q$.}
   \label{fig2}
   \end{figure}
   
 \begin{rem}
 This result is the best for the function $F_1$ for the moment, but can be optimized through applying this method on a different function, where the efficiency of our method will be much clear.
 \end{rem}
 \subsection{Proof of Theorem \ref{Thm3} }
 Recall that the Dirichlet class number formula is given by 
 \begin{equation*}
 h\left(\mathbb{Q}(\sqrt{q})\right)= \frac{\sqrt{q}}{2\log \varepsilon_q}L\left(1, \chi_q\right)
 \end{equation*}
 where $\chi_q$ is the even primitive real character modulo $q$ and $\varepsilon_q$ is the fundamental unit. Thanks to Theorem \ref{Thm1}, we find that
 \begin{equation*}
 \frac{2}{\sqrt{q}} h\left(\mathbb{Q}\left(\sqrt{q}\right) \right)
 \leq 
 \frac{\log q-0.04038}{2\log \left(\frac{1}{2}\sqrt{q-4}+\frac{1}{2}\sqrt{q}\right)}
 \leq 
 \frac{\log q-0.04038}{2\log \sqrt{q} +2\log \left(\frac{1}{2}+\frac{1}{2}\sqrt{1-\frac{4}{q}}\right)}.
 \end{equation*}
 For $q\geq 4$, we notice that 
 \begin{equation*}
 \log \left(\tfrac{1}{2}+\tfrac{1}{2}\sqrt{1-\tfrac{4}{q}}\right)=\log \left(1-\frac{2}{q\left(1+\sqrt{1-\tfrac{4}{q}}\right)}\right)\geq \log \left(1-\tfrac{2}{q}\right).
 \end{equation*}
 This yields 
 \begin{equation*}
 \frac{2}{\sqrt{q}} h\left(\mathbb{Q}\left(\sqrt{q}\right) \right)
 \leq 
 \frac{\log q-0.04038}{\log q +2\log \left(1-\frac{2}{q}\right)}.
 \end{equation*}
 Since $\log (1-X) \geq -2X$, when $ X\leq \frac{1}{2}$. It follows that
 \begin{equation*}
 \frac{2}{\sqrt{q}} h\left(\mathbb{Q}\left(\sqrt{q}\right) \right)
 \leq  
 1-\frac{2\log \left(1-\frac{2}{q}\right)+0.04038}{\log q +2\log \left(1-\frac{2}{q}\right)}
 \leq 
 1-\frac{0.04038-\frac{8}{q}}{\log q }
 \end{equation*}
 Oriat~\cite{Oriat1988} has computed the class number of this field when $1<q<24572$. For $q\geq 24572$, we conclude that 
 \begin{equation*}
 \frac{2}{\sqrt{q}} h\left(\mathbb{Q}\left(\sqrt{q}\right) \right)
 \leq 
 1-\frac{1}{25\log q}.
 \end{equation*}
 We extend it to $q>226$ via the table of Oriat.  Thanks to the precious remarks of Francesco Pappalardi, and the following simple GP-PARI code sent by him. We can check that our result is also correct for $1<q\leq 226.$
 \begin{verbatim}
 for(q=2,226,
   if(q==quaddisc(q),U=quadclassunit(q)[1];
      if(U>sqrt(q)/2*(1-1/25/log(q)), print(q" "U),),))
 \end{verbatim}
 This completes the proof.

\mytitle{Sumaia Saad Eddin}\\
Institute of Financial Mathematics\\
and Applied Number Theory,\\ Johannes Kepler University Linz,\\ Altenbergerstrasse 69, 4040 Linz, Austria.\\
{sumaia.saad\_eddin@jku.at}

 \end{document}